\documentclass{amsart}

\usepackage{amsmath}
\usepackage{amssymb}
\usepackage{amsthm}
\usepackage{mathrsfs}
\usepackage[all]{xy} 
\usepackage{bm}
\usepackage{comment}
\usepackage{color}
\usepackage{manfnt}
\usepackage{indentfirst}

\usepackage{breqn}
\usepackage{autobreak}

\usepackage{mathtools}
\mathtoolsset{showonlyrefs=true}

\newtheorem{theorem}{Theorem}[section]
\newtheorem{lemma}[theorem]{Lemma}
\newtheorem{prop}[theorem]{Proposition}
\newtheorem{corollary}[theorem]{Corollary}

\usepackage{enumerate}

\theoremstyle{definition}
\newtheorem{definition}[theorem]{Definition}

\newtheorem{example}[theorem]{Example}

\numberwithin{equation}{section}

\allowdisplaybreaks[4]

\title[Non-vanishing of MZV's]{Non-vanishing of multiple zeta values for higher genus curves over finite fields}
\author{Daichi Matsuzuki}
\email{m19044h@math.nagoya-u.ac.jp}
\address{Graduate School of Mathematics, Nagoya University, 
Furo-cho, Chikusa-ku, Nagoya, 464-8602, Japan}
\date{\today}

\begin{document}
\maketitle
\begin{abstract}
    In this paper, we show that $\infty$-adic multiple zeta values associated to the function field of an algebraic curve of higher genus over a finite field are not zero, under certain assumption on the gap sequence associated to the rational point $\infty$ on the given curve. Using arguments and results of Sheats and Thakur for the case of the projective line, we calculate the absolute values of power sums in the series defining multiple zeta values, and show that the calculation implies the non-vanishing result.
\end{abstract}
\tableofcontents

\section{Introduction}

Let $q$ be a power of a fixed prime $p$. Let $k:=\mathbb{F}_q(\theta)$ be the function field of $\mathbb{P}^{1}$ over $\mathbb{F}_q$, where $\theta$ is a variable and $1/\theta$ is the uniformizer of the infinite place $\infty$ of $\mathbb{P}^{1}$. Let $A:=\mathbb{F}_q[\theta]$ be the ring of rational functions on $\mathbb{P}^{1}$ regular away from $\infty$, and let $k_\infty$ be the completion $\mathbb{F}_q((1/\theta))$ of $k$ at the place $\infty$. In this setting,
Thakur~\cite{ThakurBook} introduced \textit{multiple zeta values} (\textit{MZV's} for short) as follows: for given positive numbers $s_1,\,\dots,\,s_d$ we put
\begin{equation}
    \zeta_A(s_1,\,\dots,\,s_d):=\sum_{\substack{a_1,\,\dots,\,a_d \in A\text{: monic}\\\deg a_1>\cdots > \deg a_d \geq0}} \frac{1}{a_1^{s_1}\cdots a_d^{s_d}} \in k_\infty. \label{DefMZVs}
\end{equation}
The number $d$ and the sum $s_1+\cdots+s_d$ are called \textit{depth} and \textit{weight} of the presentation $\zeta_{A}(s_1,\,\dots,\,s_d)$, respectively. These values are higher depth generalization of Carlitz zeta values introduced in \cite{Carlitz1935}.

Thakur~\cite{Thakur2010} proved $q$-shuffle relations of his MZV's which show that a product of two given MZV's can be expressed as $\mathbb{F}_p$-linear combination of  MZV's of the same weight. Therefore, the $k$-linear subspace of $k_\infty$ spanned by MZV's forms a $k$-algebra similarly to the case of real MZV's in the classical setting. 
Unlike the case of characteristic zero, it is not clear whether the MZV's \eqref{DefMZVs} are non-zero, 
but Thakur~\cite{Thakur2009a} proved that they are indeed non-vanishing by using the arguments and results of Sheats~\cite{Sheats1998} which he performed in order to solve Goss' positive characteristic analogue of Riemann hypothesis.

In recent years, transcendence theory for MZV's are well-developed.  Using the technique called Anderson-Brownawell-Papanikolas criterion~\cite{Anderson2004}, Chang~\cite{Chang2014} showed that the $k$-algebra of MZV's is graded (by weights). As a corollary, we can see that all MZV's are transcendental over $k$, generalizing the work of Yu~\cite{Yu1991} in the depth one case. Furthermore, Chang~\cite{Chang2016} established an effective criterion to calculate the dimensions of the $k$-vector sub-space spanned by MZV's of depth two and fixed weight. Later on, Chen and Harada~\cite{Chen2021} generalized Chang's methods to obtain lower bounds of the dimensions of $k$-linear subspaces spanned by MZV's of fixed weight and depth. Using Papanikolas' theory in terms of $t$-motives, Mishiba~\cite{Mishiba2015a, Mishiba2017} established algebraic independence results for certain families of MZV's. 
Recently, two significant conjectures addressed as Todd's conjecture and Thakur's basis conjecture, are solved independently in \cite{Chang2022} and \cite{Im2022}. The former conjecture predicts the precise dimension of the $k$-vector space spanned by MZV's of fixed weight (regarded as an analogue of Zagier's dimension conjecture for real MZV's). The latter gives an explicit $k$-basis of the subspace generated by MZV's of fixed weight  (viewed as an analogue of Hoffman's basis conjecture).

It should be emphasized that the non-vanishing result of Thakur is fundamental for these transcendence results aforementioned,  and it is the major motivation of this paper to prove non-vanishing property (Theorem \ref{main}) for the MZV's associated to the function field of an algebraic curve over a finite field $\mathbb{F}_q$ defined by Lara~Rodr\'{i}guez and Thakur under certain assumption. We believe that our results will be the first step of studying transcendence problem for MZV's associated to function fields of  higher genus curves over finite fields.

 Let $C$ be a smooth, projective and geometrically connected curve  over $\mathbb{F}_q$ with a rational point $\infty$. 
Let $k$ be the function field of $C$ over $\mathbb{F}_q$, and we denote by $A$ the ring $\Gamma( \mathcal{O}_C,\, C\setminus \{\infty\})$, which is the ring of rational functions on $C$ regular away from $\infty$. We pick a sign function $k^\times \rightarrow \mathbb{F}_q^\times$, that is a group homomorphism of multiplicative group which is an identity map on $\mathbb{F}_q^\times\subset k^\times$. We say $a\in A\setminus \{0\}$ is \textit{monic} if its sign is $1$.

The followings are the definition of MZV's associated to function fields of curves which have higher genus introduced by Lara Rodr\'{i}guez and Thakur. Let us note that even in the case $r=1$, the value introduced above does not coincides with the special value of Goss' zeta function introduced in \cite{Goss1983} unless the ring $A$ is principal.
\begin{definition}[{\cite{LaraRodriguez2021}}]\label{DefMZV}
    For positive integers $s_1,\,s_2,\,\dots,\,s_r$, we put
    \begin{equation}
        \zeta_A(s_1,\,s_2,\,\dots,\,s_r):=\sum_{\substack{a_1,\,\dots,a_r \text{:monic}\\ \deg a_1>\cdots> \deg a_r \geq0}} \frac{1}{a_1^{s_1}\cdots a_r^{s_r}} \in k_{\infty}.
    \end{equation}
    Elements in this form are called \textit{multiple zeta values}.
\end{definition}


Our main results yields non-vanishing of MZV's under the certain assumption on the non-gap sequence (see Section \ref{Pre} for the definition) as follows:
\begin{theorem}\label{main}
     Let $C$ be a smooth, projective, geometrically connected curve over $\mathbb{F}_q$ with an $\mathbb{F}_q$-rational point $\infty$. If one of the following conditions:
    \begin{enumerate}[(a)] 
        \item the non-gap sequence of $\infty$ is $\{0,\,g+1,\,g+2,\,\dots \}$, \label{condiordinary}
        \item the non-gap sequence of $\infty$ is $\{0,\,2,\,\dots,\,2g-2,\,2g,\,2g+1,\,2g+2,\,\dots \}$, \label{condihyperelliptic}
    \end{enumerate}
    is satisfied, then the multiple zeta value $\zeta_A(s_1,\,s_2,\,\dots,\,s_r)$ is non-zero for any $(s_1,\,s_2,\,\dots,\,s_r)\in \mathbb{N}^{r}$.
\end{theorem}

If the curve $C$ is an elliptic curve, then the Riemann-Roch theorem implies that conditions \eqref{condiordinary} and \eqref{condihyperelliptic} are satisfied (see~Example \ref{Exampled_i} (2)), so we have
\begin{corollary}\label{Cortomain}
MZV's associated with elliptic curves are non-zero.
\end{corollary}
Our strategy of proof is similar to that given by Thakur (\cite[Theorem 4]{Thakur2009a}).
The crucial step is the calculation of the $\infty$-adic valuation of the finite power sum $
S_{d_i}(s)$
(see Definition \ref{DefPowersum}), by which we can observe that the valuation of $S_{d_i}(s)$ is strictly increasing in $i$ for each fixed $s$, and can conclude that the absolute values of MZV's are non-zero via non-archimedean triangle inequality.
In order to calculate the valuation of the sum $S_d(s)$, we express it as series \eqref{expansion} by binomial and multinomial expansion. Using the multinonial version of Lucas' theorem (Theorem \ref{Lucas Theorem}) and applying arguments and results of Sheats and Thakur, we can find a unique term of the series which has strictly lower valuation than any other non-zero terms.

The section \ref{Pre} is devoted to review of some preliminaries needed to show our main theorem. We recall the notions as carry over base $p$, gap sequence, and power sums. We review the construction of the sequence $\{ G_i\}$ (Definition \ref{seqG}), which is used to calculate the valuation of $S_d(s)$ and also review the result Lemma \ref{FandamentalLemma} of Sheats, which played pivotal role in the proof of non-vanishing of MZV's associated to projective line by Thakur. In section \ref{proof}, we prove that the sequence $\{ G_i\}$ can be used to obtain the valuation of the power sum $S_d(s)$ also in our cases and complete the proof of Theorem \ref{main}.

\section{Preliminaries}\label{Pre}
In what follows, we fix a smooth, projective, geometrically connected curve $C$ over $\mathbb{F}_q$ with an $\mathbb{F}_q$-rational point $\infty$.
In this section, we set the notation and recall some preliminaries needed to prove our main theorem.

\begin{definition}\label{DefNonGapSeq}
We define a strictly increasing sequence  $\{d_0=0,\,d_1,\,d_2,\dots\}$ of non-negative integers to be the complement of \textit{gap sequence}, that is, the sequence characterized by
\begin{align}
    \dim L(i \infty)/L((i-1)\infty)=
    \begin{cases}
    1 & \text{$i=d_j$ for some $j$} \\
    0 & \text{otherwise},
    \end{cases}
\end{align}
for $i\geq 1$. Here, the symbol $L(i \infty)$ stands for the Riemann-Roch space, the space of elements in $A$ whose order of poles at $\infty$ is equal to or less than $i$. We call this sequence the \textit{non-gap sequence} of $C$ associated with $\infty$. This is also known as the \textit{Weierstrass semigroup}, see \cite[Definition 6.88]{Hirschfeld2008} or \cite[8-37 (a)]{Fulton1989}, for example.    
\end{definition}

By Riemann-Roch theorem, we have the following examples.
\begin{example}\label{Exampled_i}
\begin{enumerate}
    \item In the case $C={\mathbb{P}^1}$, we have $d_j=j$.
    \item If $C=E$ is an elliptic curve, then we have $d_0=0$ and $d_j=j+1$ for $j \geq1$.
\end{enumerate}
\end{example}

For each $i=0,\,1,\,2,\,\dots$, we chose a monic element in $A$ of degree $d_i$ (that is, the element in $L(d_i \infty) \setminus L((d_i-1)\infty)$) and denote it by $\xi_i$. Then each monic element $\eta$ of degree $d_i$ can be written in the form
\begin{equation}
    \eta=f_0 \xi_0+\cdots +f_{i-1} \xi_{i-1}+ \xi_i
\end{equation}
where $f_0,\,\dots,\,f_{i-1}$ are elements of $\mathbb{F}_q$ as $\eta-\xi_i$ is an element of $L((d_i-1)\infty)=L(d_{i-1}\infty)$, which has $\mathbb{F}_q$-basis $(\xi_0,\,\dotsm,\,\xi_{i-1})$.
\begin{definition}\label{DefPowersum}
For an integer $s$ and non-negative integer $d$, we put
\begin{equation}
S_d(s):=\sum\frac{1}{a^s},
\end{equation}
where $a$ runs through the set of all monic elements of degree $d$.
\end{definition}
Let us note that MZV's can be written using these power sums as follows:
\begin{align}
    \zeta_A(s_1,\,\dots,\,s_r)&=\sum_{n_1>\cdots>n_r\geq 0}S_{n_1}(s_1)\cdots S_{n_r}(s_r)\\
    &=\sum_{i_1>\cdots>i_r\geq 0}S_{d_{i_1}}(s_1)\cdots S_{d_{i_r}}(s_r)
    
\end{align}

Let us consider non-negative integers $n,\,m_0,\,\dots,\,m_i$ such that $n=m_0+\cdots+m_i$ and their $p$-adic expansions
\begin{equation}
    n=\sum_{l \geq0} n_l p^l,\quad m_j=\sum_{l \geq 0}m_{j,\,l} p^{l}.
\end{equation}
\begin{definition}We say that the summation $n=m_0+\cdots+m_i$ \textit{has no carry over base $p$} if $m_{0,\,l}+\cdots+m_{i,\,l} \leq p-1$ for each $l$. 
\end{definition}

We need two properties of binomial and multinomial coefficients. We define multinomial coefficients by the generating series as follows:
\begin{equation}
    (X_0+\cdots+X_i)^n=\sum_{m_0,\,\dots,\,m_r \in \,\mathbb{Z}}\binom{n}{m_0,\,\dots,\,m_i}X_0^{m_0}\cdots X_i^{m_i}.
\end{equation}
The first one is the formula
\begin{equation}
    \binom{-s}{y}=(-1)^y\binom{y-s-1}{y} \label{binomprop1}
\end{equation}
which holds for any integer $y$, see \cite[(5.14)]{GrahamBook} for example.
The second one is the following multinomial version of famous Lucas' Theorem:
\begin{theorem}[Lucas' Theorem]\label{Lucas Theorem}
The congruence
\begin{equation}
    \binom{n}{m_0,\,\dots,\,m_i} \equiv \prod_{l \geq 0} \binom{n_i}{m_{0,\,l},\,\dots,\,m_{i,\,l}} \mod p
\end{equation}
    holds. 
\end{theorem}
By the above congruence, one has the following criterion on vanishing of multinomial coefficients, which is well-known by experts but we give a short proof in order to make the present paper self-contained.
\begin{lemma}\label{multinomialcriterion}
    The multinomial coefficient $$\binom{n}{m_0,\,\dots,\,m_i}$$ is non-zero in $\mathbb{F}_p$ if and only if the sum $n=m_0+\cdots+m_i$ has no carry over base $p$.
\end{lemma}
\begin{proof}
    Let us suppose the sum $n=m_0+\cdots+m_i$ has carry over base $p$ and take minimal $l$ such that $m_{0,\,l}+\cdots+m_{i,\,l} > p-1$; then we have $m_{0,\,l}+\cdots+m_{i,\,l} > n_l$ as $0\leq n_l \leq p-1$ and have $$\binom{n_l}{m_{0,\,l},\,\dots,\,m_{i,\,l}}=0$$ by definition. Hence we obtain
    \begin{equation}
        \binom{n}{m_0,\,\dots,\,m_i} \equiv 0 \mod p
    \end{equation}Let us suppose conversely that the sum $n=m_0+\cdots+m_i$ has no carry over base $p$; then, for each $l$, we have $n_l=m_{0,\,l}+\cdots+m_{i,\,l}$ and each factor
    \begin{equation}
        \binom{n_i}{m_{0,\,l},\,\dots,\,m_{i,\,l}}=\frac{n_l!}{m_{0,\,l}!\cdots,m_{i,\,l}!}
    \end{equation}
    in the congruence of Lucas is a non zero integer prime to $p$ as $n_l,\,m_{0,\,l},\,\dots,\,m_{i,\,l} \leq p-1$. Hence the multinomial coefficient in not zero in $\mathbb{F}_p$.
\end{proof}

Let us recall the following fundamental identity from the theory of characters over finite fields:
\begin{equation}
    \sum_{f \in \mathbb{F}_q} f^m=
    \begin{cases}
        -1 & \text{ if $m$ is a positive multiple of $q-1$},\\
        0 & \text{if $m$ is non-negative integer not divisible by $q-1$}.
    \end{cases}\label{PPFF}
\end{equation}

\begin{definition}\label{DefWeightedSum}
 Let   $\{d_0=0,\,d_1,\,d_2,\dots\}$ be the non-gap sequence of $C$ associated with $\infty$ defined in Definition \ref{DefNonGapSeq}. For any positive integers $m_0,\,\dots,\,m_{i-1}$, we define
\begin{equation}
    \operatorname{WS}_C(m_0,\,\dots,\,m_{i-1}):=
    \left(d_i-d_0\right)m_0+\cdots+\left(d_i-d_{i-1}\right)m_{i-1},
\end{equation}
(here WS stands for weighted sum).
  
\end{definition}

For example, if $C$ is the projective line, then \begin{equation}
        \operatorname{WS}_{\mathbb{P}^1}(m_0,\,\dots,\,m_{i-1})=im_0+(i-1)m_1+\cdots+m_{i-1}.
\end{equation} 
by Example \ref{Exampled_i}. If $C$ is an elliptic curve, then 
\begin{equation}
    \operatorname{WS}_C(m_0,\,\dots,\,m_{i-1})=(i+1)m_0+(i-1)m_1+(i-2)m_2+\cdots+m_{i-1}.
\end{equation}

The following observations are a useful examples and will be used in the proof of our main result. 

\begin{example}

Suppose that the curve $C$ has a genus $g\geq 1$. If the non-gap sequence of $C$ associated with $\infty$ is $\{0,\,g+1,\,g+2,\,\dots \}$, then it follows that
\begin{align}
    \operatorname{WS}_C(m_0,\,\dots,\,m_{i-1})&=(i+g)m_0+(i-1)m_1+(i-2)m_2+\cdots+m_{i-1}\\
    &=\operatorname{WS}_{\mathbb{P}^1}(m_0,\,\dots,\,m_{i-1})+gm_0.
\end{align}
and if the non-gap sequence is $\{0,\,2,\,\dots,\,2g-2,\,2g,\,2g+1,\,2g+2,\,\dots \}$, then we obtain   
\begin{align}
    \operatorname{WS}_C(m_0,\,\dots,\,m_{i-1})&=(i+g)m_0+(i+g-2)m_1+\cdots+(i-g+2)m_{g-1}\\
    &\quad \quad+(i-g)m_g+(i-g-1)m_{g+1}+\cdots+m_{i-1}\\
    &= \operatorname{WS}_{\mathbb{P}^1}(m_0,\,\dots,\,m_{i-1})+\operatorname{WS}_{\mathbb{P}^1}(m_0,\,\dots,\,m_{g-1})
\end{align}
for $i \geq g$, and
\begin{equation}
\operatorname{WS}_C(m_0,\,\dots,\,m_{i-1})=2\operatorname{WS}_{\mathbb{P}^1}(m_0,\,\dots,\,m_{i-1})
\end{equation}
for $i<g$.
\end{example}

The sequence introduced below is used to calculate the valuation of the power sum $S_d(s)$.
\begin{definition}\label{seqG}
    For a positive integer $s$, we define the following sequence $\{G_i \}_{i=0,\,1,\,2\,\dots}$, which depends only on $q$ and $s$, but not on the choice of curve $C$:
\begin{quote}
    for each $i\geq 0$, the  number $G_i$ is the smallest positive multiple of $q-1$ such that the summation 
\begin{equation}
(s-1)+G_0+G_1+G_2+\cdots+G_i
\end{equation}
has no carry over base $p$.  
\end{quote}
\end{definition}

\begin{example}
    Let us consider the case where $q=7$ and $s=223413_{(7)}:=2\cdot7^5+2\cdot 7^4+3\cdot 7^3+4 \cdot 7^2 +1 \cdot 7^1+3\cdot7^0$. As each positive integer is multiple of $q-1$ if and only if so is the sum of digits of its $q$-adic expansion, we have $G_0=24_{(7)}$ and $s-1+G_0=223436_{(7)}$, so we obtain $G_1=1230_{(7)}$ and $s-1+G_0+G_1=224666_{(7)}$. We can continue to obtain $G_2=42000_{(7)}$, $G_3=2400000_{(7)}$, $G_4=24000000_{(7)}$, and so on (in this case the sum of digits of $q$-adic expansion of each $G_i$ is $q-1=6$ by the minimality).
    If $s=3251321_{(7)}$, then we have $G_0=6_{(7)}$, $G_1=240_{(7)}$, $G_2=5100_{(7)}$, $G_3=1410000_{(7)}$, $G_4=42000000_{(7)}$, and so on.
\end{example}

The following lemma is essentially obtained by Sheats and plays a pivotal role in Thakur's proof for his non-vanishing result of MZV's associated to the function field of $\mathbb{P}^{1}$ over $\mathbb{F}_q$.

\begin{lemma}[{\cite{Sheats1998}, \cite[Theorem 1]{Thakur2009a}}]  \label{FandamentalLemma}
Suppose that the tuple $(m_0,\,\dots,\,m_{i-1})$ of positive integers satisfies the following three conditions:
\begin{enumerate}
\item
each of $m_0,\,\dots,\,m_{i-1}$ is a multiple of $q-1$;\label{even condition}
\item \label{condiLucas}
the sum $(s-1)+ m_0+\cdots+m_{i-1}$ has no carry over base $p$;\label{Lucas condition}
\item we have $(m_0,\,\dots,\,m_{i-1}) \neq (G_0,\,\dots,\,G_{i-1})$. \label{conditionnoneq}
\end{enumerate}
Then we have the following strict inequality
\begin{equation}
   \operatorname{WS}_{\mathbb{P}^1}(G_0,\,\dots,\,G_{i-1})<\operatorname{WS}_{\mathbb{P}^1}(m_0,\,\dots,\,m_{i-1}).
\end{equation}
\end{lemma}

In the next section, we give a proof of our main result, and the lamma above plays an important role in our arguments.

\section{Proof of non-vanishing} \label{proof}
In this section, we continue with setting that $C$ is a smooth,  projective and geometrically connected curve over $\mathbb{F}_q$ with an $\mathbb{F}_q$-rational point $\infty$, and our goal is to prove Theorem \ref{main} and Corollary \ref{Cortomain}. We still denote by $\left\{ d_0,\,d_1,\,\dots\right\}$ the non-gap sequence of $C$ associated with $\infty$ (see Definition \ref{DefNonGapSeq}), and put
\begin{equation}
    \operatorname{WS}_C(m_0,\,\dots,\,m_{i-1}):=
    \left(d_i-d_0\right)m_0+\cdots+\left(d_i-d_{i-1}\right)m_{i-1},
\end{equation}
for positive integers $m_0,\,\dots,\,m_{i-1}$. For each $\alpha \in k_\infty$, we denote its $\infty$-adic valuation by $v_\infty(\alpha)$.

\begin{prop}
For positive integers $s$ and $i$, 
 if one  of the two conditions
    \begin{enumerate}[(a)]
        \item the non-gap sequence of $\infty$ is $\{0,\,g+1,\,g+2,\,\dots \}$,
        \item the non-gap sequence of $\infty$ is $\{0,\,2,\,\dots,\,2g-2,\,2g,\,2g+1,\,2g+2,\,\dots \}$,
    \end{enumerate}
    is satisfied, then we have
    \begin{align}
   v_\infty \left(S_{d_i}(s)\right)
    =d_is +\operatorname{WS}_C(G_0,\,\dots,\,G_{i-1}),
\end{align} 
where the power sum $S_{d_i}(s)$ is the power sum given in Definition \ref{DefPowersum} and $\{G_0, G_1, \ldots\}$ are defined in Definition \ref{seqG}.
\end{prop}

\begin{proof}
We continue to denote by $\xi_i\in L(d_i \infty) \setminus L((d_i-1)\infty)$ the fixed monic element of degree $d_i$.
As each monic element of degree $d_i$ can be uniquely written in the form
\begin{equation}
    f_0 \xi_0+\cdots +f_{i-1} \xi_{i-1}+ \xi_i
\end{equation}
with $f_0,\,\dots,\,f_{i-1}\in \mathbb{F}_q$, we have
\begin{align}
 S_{d_i}(s)&=\sum_{f_0,\,\dots,\,f_{i-1} \in \mathbb{F}_q} (f_0 \xi_0+\cdots+f_{i-1}\xi_{i-1}+\xi_i)^{-s}\\
&=\xi_i^{-s} \sum_{f_0,\,\dots,\,f_{i-1} \in \mathbb{F}_q} \left(f_0 \frac{\xi_0}{\xi_i}+\cdots+f_{i-1}\frac{\xi_{i-1}}{\xi_i}+1\right)^{-s}\\
\intertext{(using binomial expansion)}
&=\xi_i^{-s}  \sum_{f_0,\dots,\,f_{i-1}  \in \mathbb{F}_q}\  \sum_{y\geq 0} \binom{-s}{y}\left(f_0 \frac{\xi_0}{\xi_i}+\cdots+f_{i-1}\frac{\xi_{i-1}}{\xi_i}\right)^{y}\\
\intertext{(using multinomial expansion)}
&=\xi_i^{-s} \sum_{f_0,\dots,\,f_{i-1}  \in \mathbb{F}_q} \sum_{y\geq 0} \binom{-s}{y} \sum_{m_0,\dots,\,m_{i-1} \geq0} \binom{y}{m_0,\dots,\,m_{i-1} }\cdot \\
&\quad\ \ \, \cdot \left(f_0 \frac{\xi_0}{\xi_i}\right)^{m_0}\cdots\left(f_{i-1}\frac{\xi_{i-1}}{\xi_i}\right)^{m_{i-1}}\\
&=\sum_{y\geq 0} \ \sum_{m_0,\dots,\,m_{i-1} \geq0} \xi_i^{-s} \left(\frac{\xi_0}{\xi_i}\right)^{m_0}\cdots\left(\frac{\xi_{i-1}}{\xi_i}\right)^{m_{i-1}} \binom{-s}{y}
 \cdot \\
&\quad\ \ \, \cdot\binom{y}{m_0,\dots,\,m_{i-1} } \sum_{f_0,\dots,\,f_{i-1}  \in \mathbb{F}_q}f_0 ^{m_0}\cdots f_{i-1}^{m_{i-1}}\\
&=\sum_{y\geq 0} \ \sum_{m_0,\dots,\,m_{i-1} \geq0} \xi_i^{-s} \left(\frac{\xi_0}{\xi_i}\right)^{m_0}\cdots\left(\frac{\xi_{i-1}}{\xi_i}\right)^{m_{i-1}}(-1)^y\binom{s-1+y}{y}
 \cdot \label{expansion}\\
&\quad\ \ \, \cdot \binom{y}{m_0,\dots,\,m_{i-1} } \sum_{f_0,\dots,\,f_{i-1}  \in \mathbb{F}_q}f_0 ^{m_0}\cdots f_{i-1}^{m_{i-1}},  \\
\end{align}
where the last equality is from \eqref{binomprop1}.
Here, by Lemma \ref{multinomialcriterion}, the quantity
\begin{equation}
\binom{s-1+y}{y}\binom{y}{m_0\,\dots,\,m_{i-1} }=\binom{s-1+y}{s-1,\,m_0\,\dots,\,m_{i-1}}
\end{equation}
is non-zero if and only if the tuple $(m_0,\,\dots,\,m_{i-1})$ satisfies condition \eqref{Lucas condition} in Lemma \ref{FandamentalLemma}, and the factor
\begin{equation}
\sum_{f_0,\dots,\,f_{i-1}  \in \mathbb{F}_q}f_0 ^{m_0}\cdots f_{i-1}^{m_{i-1}}
\end{equation}
of the term is nonzero if and only if the tuple $(m_0,\,\dots,\,m_{i-1})$ satisfies condition \eqref{even condition} in Lemma \ref{FandamentalLemma}.
We note that the choice $(m_0,\,m_1,\,\dots,\,m_{i-1})=(G_0,\,G_1,\,\dots,\,G_{i-1})$, whihc is defined in Definition \ref{seqG}, satisfies conditions \eqref{even condition} and \eqref{Lucas condition} of Lemma \ref{FandamentalLemma}. If these two conditions are satisfied, the $\infty$-adic valuation of the term
\begin{align}
    \xi_i^{-s} \left(\frac{\xi_0}{\xi_i}\right)^{m_0}\cdots\left(\frac{\xi_{i-1}}{\xi_i}\right)^{m_{i-1}}(-1)^y\binom{s-1+y}{y}
 \cdot \\
\quad\ \ \, \cdot \binom{y}{m_0,\dots,\,m_{i-1} } \sum_{f_0,\dots,\,f_{i-1}  \in \mathbb{F}_q}f_0 ^{m_0}\cdots f_{i-1}^{m_{i-1}},
\end{align}
is equal to
\begin{align}
    &\ \ \ d_is+\left(d_i-d_0\right)m_0+\cdots+\left(d_{i}-d_{i-1}\right)m_{i-1}\\
    &=d_is +\operatorname{WS}_C(m_0,\,\dots,\,m_{i-1}).
\end{align}

Let us take a tuple $(m_1,\,\dots,\,m_{i-1})$ satisfying \eqref{even condition},  \eqref{condiLucas}, and \eqref{conditionnoneq} in Lemma \ref{FandamentalLemma}. If the condition \eqref{condiordinary} is satisfied, then we have that
\begin{align}
    &\quad \  d_is+\operatorname{WS}_C(m_1,\,\dots,\,m_{i-1})-\big(d_is +\operatorname{WS}_C(G_0,\,\dots,\,G_{i-1})\big)\\
    &=\operatorname{WS}_C(m_0,\,\dots,\,m_{i-1})-\operatorname{WS}_C(G_0,\,\dots,\,G_{i-1})\\
    &=\operatorname{WS}_{\mathbb{P}^1}(m_0,\,\dots,\,m_{i-1})-\operatorname{WS}_{\mathbb{P}^1}(G_0,\,\dots,\,G_{i-1})+g(m_0-G_0)>0
\end{align}
as $\operatorname{WS}_{\mathbb{P}^1}(G_0,\,\dots,\,G_{i-1})<\operatorname{WS}_{\mathbb{P}^1}(m_0,\,\dots,\,m_{i-1})$ and $G_0\leq m_0$ by Lemma \ref{FandamentalLemma}. On the other hand, if the condition \eqref{condihyperelliptic} is satisfied, then
\begin{align}
    &\quad \  d_is+\operatorname{WS}_C(m_1,\,\dots,\,m_{i-1})-\big(d_is+ \operatorname{WS}_C(G_0,\,\dots,\,G_{i-1})\big)\\
    &=\operatorname{WS}_C(m_0,\,\dots,\,m_{i-1})-\operatorname{WS}_C(G_0,\,\dots,\,G_{i-1})\\
    &=\operatorname{WS}_{\mathbb{P}^1}(m_0,\,\dots,\,m_{i-1})-\operatorname{WS}_{\mathbb{P}^1}(G_0,\,\dots,\,G_{i-1})\\
    &\quad \quad +\operatorname{WS}_{\mathbb{P}^1}(m_0,\,\dots,\,m_{i^\prime-1})-\operatorname{WS}_{\mathbb{P}^1}(G_0,\,\dots,\,G_{i^\prime-1})>0
\end{align}
where $i^\prime=\operatorname{min}(i,\,g)$ as $\operatorname{WS}_{\mathbb{P}^1}(G_0,$ $\dots,\,G_{i-1})<\operatorname{WS}_{\mathbb{P}^1}(m_0,\,\dots,\,m_{i-1})$ and $\operatorname{WS}_{\mathbb{P}^1}(G_0,$ $\dots,$ $G_{i^\prime-1})\leq \operatorname{WS}_{\mathbb{P}^1}(m_0,\,\dots,\,m_{i^\prime-1})$ by Lemma \ref{FandamentalLemma}.

Therefore,
\begin{align}
    & \ \ \ v_\infty \left(S_{d_i}(s)\right)\\
    &= v_\infty \Biggl( \xi_i^{-s} \left(\frac{\xi_0}{\xi_i}\right)^{G_0}\cdots\left(\frac{\xi_{i-1}}{\xi_i}\right)^{G_{i-1}}(-1)^y\binom{s-1+y}{y}
 \cdot \\
&\quad\ \ \, \cdot \binom{y}{G_0,\dots,\,G_{i-1} } \sum_{f_0,\dots,\,f_{i-1}  \in \mathbb{F}_q}f_0 ^{G_0}\cdots f_{i-1}^{G_{i-1}} \Biggr)\\
&=d_is +\operatorname{WS}_C(G_0,\,\dots,\,G_{i-1}).
\end{align}
\end{proof}

{\bf{Proof of Theorem \ref{main}}}. Now, we give a proof of Theorem \ref{main}.
By this proposition, we have
\begin{equation}
    v_\infty \left(S_{d_{i-1}}(s)\right)-v_\infty \left(S_{d_i}(s)\right)
    =(d_{i-1}-d_{i})(s+G_0+\cdots+G_{i-1})<0.
\end{equation}
Hence we obtain
    \begin{align}
        v_\infty \Bigl(\zeta_A(s_1,\,\dots,\,s_r)\Bigr)=v_\infty \Bigl(S_{d_{r-1}}(s_1)S_{d_{r-2}}(s_2)\cdots S_{d_0}(s_r)\Bigr)<\infty.
    \end{align}
Therefore, the multiple zeta value in question is nonzero as desired. \qed

\vskip.5\baselineskip

\noindent\textbf{Acknowledgement}.
The author would like to thank his advisor H.~Furusho for his profound instruction and continuous encouragements. He is also grateful to C.-Y.~Chang for suggesting the topic treated in this paper. This work was {\color{black}partially} supported by JST SPRING, Grant Number JPMJSP2125, {\color{black} and C.-Y. Chang's NSTC Grant 111-2628-M-007-002-}.



\end{document}